\documentclass[reqno,oneside]{amsart}

\usepackage[a4paper, hmargin={3.5cm, 3.5cm}]{geometry}  
\usepackage[danish,english]{babel} 
\usepackage[latin1]{inputenc} 
\usepackage[T1]{fontenc} 
\usepackage{lmodern}
\usepackage{amsmath,amsthm,amssymb,amsfonts,mathrsfs,latexsym} 
\usepackage{graphicx} 
\usepackage{verbatim} 
\usepackage[all]{xy} 
\usepackage[hyperfootnotes=false, plainpages=false, pdfpagelabels, colorlinks, linkcolor=red, citecolor=blue, urlcolor=blue, hypertexnames=true,hidelinks]{hyperref} 
\usepackage{multirow}

\makeatletter
\newtheorem*{rep@theorem}{\rep@title}
\newcommand{\newreptheorem}[2]{%
\newenvironment{rep#1}[1]{%
 \def\rep@title{#2 \ref{##1}}%
 \begin{rep@theorem}}%
 {\end{rep@theorem}}}
\makeatother

\newtheorem{thm}{Theorem}[section]
\newtheorem{thmA}{Theorem}
\newreptheorem{thm}{Theorem}
\newtheorem{lem}[thm]{Lemma}
\newtheorem{prop}[thm]{Proposition}

\newtheorem*{thm*}{Theorem}
\newtheorem*{claim*}{Claim}

\theoremstyle{definition}
\newtheorem{defi}[thm]{Definition}

\newtheorem{rem}[thm]{Remark}

\newcommand{\mc}[1]{\mathcal{#1}}
\newcommand{\ms}[1]{\mathscr{#1}}
\newcommand{\mb}[1]{\mathbb{#1}}
\newcommand{\mf}[1]{\mathfrak{#1}}
\newcommand{\mr}[1]{\mathrm{#1}}

\newcommand{\C}{\mathbb{C}}

\newcommand{\R}{\mathbb{R}}
\newcommand{\Z}{\mathbb{Z}}

\newcommand{\la}{\langle}
\newcommand{\ra}{\rangle}
\renewcommand{\epsilon}{\varepsilon}
\renewcommand{\phi}{\varphi}
\renewcommand{\tilde}{\widetilde}

\newcommand{\cb}{\mr{cb}}

\DeclareMathOperator{\GL}{GL}
\DeclareMathOperator{\SL}{SL}
\DeclareMathOperator{\SO}{SO}
\DeclareMathOperator{\SU}{SU}
\DeclareMathOperator{\Sp}{Sp}
\DeclareMathOperator{\Spin}{Spin}

\DeclareMathOperator{\id}{id}

\newcommand{\FF}{\mr F_{4(-20)}}
\newcommand{\WA}{\mr{WA}}
\newcommand{\WH}{\mr{WH}}

\numberwithin{equation}{section}
\begin{document}
\selectlanguage{english}


\title{The weak {H}aagerup property II: {E}xamples}
\author{Uffe Haagerup}
\thanks{Both authors are supported by ERC Advanced Grant no.~OAFPG 247321 and the Danish National Research Foundation through the Centre for Symmetry and Deformation (DNRF92). The first author is also supported by the Danish Natural Science Research Council.}
\address{Department of Mathematical Sciences, University of Copenhagen,
\newline Universitetsparken 5, DK-2100 Copenhagen \O, Denmark}
\email{haagerup@math.ku.dk}

\author{S{\o}ren Knudby}
\address{Department of Mathematical Sciences, University of Copenhagen,
\newline Universitetsparken 5, DK-2100 Copenhagen \O, Denmark}
\email{knudby@math.ku.dk}
\date{September 20, 2014}

\begin{abstract}
The weak Haagerup property for locally compact groups and the weak Haagerup constant was recently introduced by the second author in \cite{K-WH}. The weak Haagerup property is weaker than both weak amenability introduced by Cowling and the first author in \cite{MR996553} and the Haagerup property introduced by Connes \cite{Connes-Kingston} and Choda \cite{MR718798}. 

In this paper it is shown that a connected simple Lie group $G$ has the weak Haagerup property if and only if the real rank of $G$ is zero or one. Hence for connected simple Lie groups the weak Haagerup property coincides with weak amenability. Moreover, it turns out that for connected simple Lie groups the weak Haagerup constant coincides with the weak amenability constant, although this is not true for locally compact groups in general.

It is also shown that the semidirect product $\R^2\rtimes\SL(2,\R)$ does not have the weak Haagerup property.
\end{abstract}
\maketitle
\setcounter{tocdepth}{1}

\section{Introduction}
Amenability is a fundamental concept for locally compact groups, see, for example, the books \cite{MR0251549}, \cite{MR767264}. In the 1980s, two weaker properties for locally compact groups were introduced, first the \emph{Haagerup property} by Connes \cite{Connes-Kingston} and Choda \cite{MR718798} and next \emph{weak amenability} by Cowling and the first author \cite{MR996553}. Both properties have been studied extensively (see \cite[Chapter~12]{MR2391387}, \cite{MR1852148} and \cite{MR2132866} and the references therein). It is well known that amenability of a locally compact group $G$ is equivalent to the existence of a net $(u_\alpha)_{\alpha\in A}$ of continuous, compactly supported, positive definite functions on $G$ such that $(u_\alpha)_{\alpha\in A}$ converges to the constant function $1_G$ uniformly on compact subsets of $G$.

\begin{defi}[\cite{Connes-Kingston},\cite{MR1852148}]
A locally compact group $G$ has the \emph{Haagerup property} if there exists a net $(u_\alpha)_{\alpha\in A}$ of continuous positive definite functions on $G$ vanishing at infinity such that $u_\alpha\to 1_G$ uniformly on compact sets.
\end{defi}

As usual we let $C_0(G)$ denote the continuous (complex) functions on $G$ vanishing at infinity and let $C_c(G)$ be the subspace of functions with compact support. Also, $B_2(G)$ denotes the space of Herz--Schur multipliers on $G$ with the Herz--Schur norm $\|\ \|_{B_2}$ (see Section~\ref{sec:prelim} for more details).

\begin{defi}[\cite{MR996553}]
A locally compact group $G$ is \emph{weakly amenable} if there exist a constant $C > 0$ and a net $(u_\alpha)_{\alpha\in A}$ in $B_2(G)\cap C_c(G)$ such that
\begin{align}\label{eq:WA-bound}
\|u_\alpha\|_{B_2} \leq C  \quad\text{ for every } \alpha\in A,
\\
u_\alpha\to 1 \text{ uniformly on compacts}.
\end{align}
The best possible constant $C$ in \eqref{eq:WA-bound} is called the \emph{weak amenability constant} denoted $\Lambda_\WA(G)$. If $G$ is not weakly amenable, then we put $\Lambda_\WA(G) = \infty$. The weak amenability constant $\Lambda_\WA(G)$ is also called the Cowling--Haagerup constant and denoted $\Lambda_\cb(G)$ or $\Lambda_G$ in the literature.
\end{defi}

The definition of weak amenability given here is different from the definition given in \cite{MR996553}, but the definitions are equivalent. In one direction, this follows from \cite[Proposition~1.1]{MR996553} and the fact that $A(G)\subseteq B_2(G)$, where $A(G)$ denotes the Fourier algebra of $G$ (see Section~\ref{sec:prelim}). In the other direction, one can apply the convolution trick (see \cite[Appendix~B]{K-WH}) together with the standard fact that $C_c(G) * C_c(G) \subseteq A(G)$.

\begin{defi}[\cite{MR3146826},\cite{K-WH}]
A locally compact group $G$ has the \emph{weak Haagerup property} if there exist a constant $C > 0$ and a net $(u_\alpha)_{\alpha\in A}$ in $B_2(G)\cap C_0(G)$ such that
\begin{align}\label{eq:WH-bound}
\|u_\alpha\|_{B_2} \leq C  \quad\text{ for every } \alpha\in A,
\\
u_\alpha\to 1 \text{ uniformly on compacts}.
\end{align}
The best possible constant $C$ in \eqref{eq:WH-bound} is called the \emph{weak Haagerup constant} denoted $\Lambda_\WH(G)$. If $G$ does not have the weak Haagerup property, then we put $\Lambda_\WH(G) = \infty$.
\end{defi}

Clearly, the weak Haagerup property is weaker than both the Haagerup property and weak amenability, and hence there are many known examples of groups with the weak Haagerup property. Moreover, there exist examples of groups that fail the first two properties but nevertheless have the weak Haagerup property (see \cite[Corollary~5.7]{K-WH}).

Our first result is the following theorem.
\begin{thmA}\label{thm:SL3}
The groups $\SL(3,\R)$, $\Sp(2,\R)$, and $\tilde\Sp(2,\R)$ do not have the weak Haagerup property.
\end{thmA}

The case of $\SL(3,\R)$ can also be found in \cite[Theorem~5.1]{MR2838352} (take $p=\infty$). Our proof of Theorem~\ref{thm:SL3} is a fairly simple application of the recent methods and results of de~Laat and the first author \cite{MR3047470}, \cite{delaat2}, where it is proved that a connected simple Lie group $G$ of real rank at least two does not have the Approximation Property (AP), that is, there is no net $(u_\alpha)_{\alpha\in A}$ in $B_2(G)\cap C_c(G)$ which converges to the constant function $1_G$ in the natural weak$^*$-topology on $B_2(G)$. By inspection of their proofs in the case of the three groups mentioned in Theorem~\ref{thm:SL3}, one gets that for those three groups the net $(u_\alpha)_{\alpha\in A}$ cannot even be chosen as functions in $B_2(G)\cap C_0(G)$, which proves Theorem~\ref{thm:SL3}.

By standard structure theory of connected simple Lie groups, it now follows that the conclusion of Theorem~\ref{thm:SL3} holds for all connected simple Lie groups of real rank at least two. Moreover, by \cite{MR748862}, \cite{MR996553}, \cite{MR784292}, \cite{MR1079871} every connected simple Lie group of real rank zero or one is weakly amenable. We thus obtain the following theorem.

\begin{thmA}\label{thm:simple-Lie-groups}
Let $G$ be a connected simple Lie group. Then $G$ has the weak Haagerup property if and only if the real rank of $G$ is at most one.
\end{thmA}

For connected simple Lie groups $G$ the constants $\Lambda_\WA(G)$ are known: if the real rank is zero, then $G$ is compact and $\Lambda_\WA(G) = 1$. If the real rank is at least two, then by \cite{UH-preprint}, \cite{MR1418350} the group $G$ is not weakly amenable and hence $\Lambda_\WA(G) = \infty$. Finally, in the real rank one case, one has by \cite{MR748862}, \cite{MR996553}, \cite{MR784292}, \cite{MR1079871} that
\begin{align}\label{eq:WA}
\Lambda_{\WA}(G) = \begin{cases}
1 & \text{ for } G\approx\SO_0(1,n) \\
1 & \text{ for } G\approx\SU(1,n) \\
2n-1 & \text{ for } G\approx\Sp(1,n) \\
21 & \text{ for } G\approx\FF \\
\end{cases}
\end{align}
where $G \approx H$ means that $G$ is locally isomorphic to $H$. We prove the following theorem.
\begin{thmA}\label{thm:rank-one}
For every connected simple Lie group $G$, $\Lambda_\WA(G) = \Lambda_\WH(G)$.
\end{thmA}

It is clear that for every locally compact group $G$ one has $1\leq\Lambda_\WH(G)\leq\Lambda_\WA(G)$. Theorem~\ref{thm:rank-one} then amounts to show that $\Lambda_\WH(G) = \Lambda_\WA(G)$ when $G$ is locally isomorphic to $\Sp(1,n)$ or $\FF$. Moreover, since the groups $\Sp(1,n)$ and $\FF$ are simply connected and have finite center, one can actually restrict to the case when $G$ is either $\Sp(1,n)$ or $\FF$. The proof of Theorem~\ref{thm:rank-one} in these two cases relies heavily on a result from \cite{K-fourier}, namely that for $\Sp(1,n)$ and $\FF$ the minimal parabolic subgroup $P = MAN$ of these groups has the property that $A(P) = B(P)\cap C_0(P)$. Here, $A(P)$ and $B(P)$ denote, respectively, the Fourier algebra and the Fourier--Stieltjes algebra of $P$ (see Section~\ref{sec:prelim}).

For all the groups mentioned so far, the weak Haagerup property coincides with weak amenability and with the AP. As an example of a group with the AP which fails the weak Haagerup property we have the following theorem.
\begin{thmA}\label{thm:SL2}
The group $\R^2 \rtimes \SL(2,\R)$ does not have the weak Haagerup property.
\end{thmA}
Combining Theorem~\ref{thm:SL2} with \cite[Theorem~A]{K-WH} we observe that the discrete group $\Z^2\rtimes\SL(2,\Z)$, which is a lattice in $\R^2\rtimes\SL(2,\R)$, also does not have the weak Haagerup property.

Theorem~\ref{thm:SL2} generalizes a result from \cite{UH-preprint} where it is shown that $\R^2\rtimes\SL(2,\R)$ is not weakly amenable. Crucial to our proof of Theorem~\ref{thm:SL2} are some of the techniques developed in \cite{UH-preprint}. These techniques are further developed here using a result from \cite{MR1220905}, namely that $\R^2\rtimes\SL(2,\R)$ satisfies the AP. Also, \cite[Theorem~2]{K-fourier} is essential in the proof of Theorem~\ref{thm:SL2}.

Both groups $\R^2$ and $\SL(2,\R)$ enjoy the Haagerup property and hence also the weak Haagerup property. Theorem~\ref{thm:SL2} thus shows that extensions of groups with the (weak) Haagerup property need not have the weak Haagerup property.


\section{Preliminaries}\label{sec:prelim}
Let $G$ be a locally compact group equipped with a left Haar measure. We denote the left regular representation of $G$ on $L^2(G)$ by $\lambda$. As usual, $C(G)$ denotes the (complex) continuous functions on $G$. When $G$ is a Lie group, $C^\infty(G)$ is the space of smooth functions on $G$.


We first describe the Fourier--Stieltjes algebra and the Fourier algebra of $G$. These were originally introduced in the seminal paper \cite{MR0228628} to which we refer for further details about these algebras. Afterwards we describe the Herz--Schur multiplier algebra.

The \emph{Fourier--Stieltjes algebra} $B(G)$ can be defined as set of matrix coefficients of strongly continuous unitary representations of $G$, that is, $u\in B(G)$ if and only if there are a strongly continuous unitary representation $\pi\colon G\to U(\mc H)$ of $G$ on a Hilbert space $\mc H$ and vectors $x,y\in \mc H$ such that
\begin{align}\label{eq:matrix-coefficient}
u(g) = \la \pi(g)x,y\ra \quad\text{for all } g\in G.
\end{align}
The norm $\|u\|_B$ of $u\in B(G)$ is defined as the infimum (actually a minimum) of all numbers $\|x\|\|y\|$, where $x,y$ are vectors in some representation $(\pi,\mc H)$ such that \eqref{eq:matrix-coefficient} holds. With this norm $B(G)$ is a unital Banach algebra. The Fourier--Stieltjes algebra coincides with the linear span of the continuous positive definite functions on $G$. For any $u\in B(G)$ the inequality $\|u\|_\infty\leq \|u\|_B$ holds, where $\|\ \|_\infty$ denotes the uniform norm.

The compactly supported functions in $B(G)$ form an ideal in $B(G)$, and the closure of this ideal is the \emph{Fourier algebra} $A(G)$, which is then also an ideal. The Fourier algebra coincides with the set of matrix coefficients of the left regular representation $\lambda$, that is, $u\in A(G)$ if and only if there are vectors $x,y\in L^2(G)$ such that
\begin{align}\label{eq:matrix-coefficient-regular}
u(g) = \la \lambda(g)x,y\ra \quad\text{for all } g\in G.
\end{align}
The norm of $u\in A(G)$ is the infimum of all numbers $\|x\|\|y\|$, where $x,y\in L^2(G)$ satisfy \eqref{eq:matrix-coefficient-regular}. We often write $\|u\|_A$ for the norm $\|u\|_B$ when $u\in A(G)$.

The dual space of $A(G)$ can be identified with the group von Neumann algebra $L(G)$ of $G$ via the duality
$$
\la a,u\ra = \la ax,y \ra = \int_G (ax)(g) \overline{y(g)}\ dg
$$
where $a\in L(G)$ and $u\in A(G)$ is of the form \eqref{eq:matrix-coefficient-regular}.

When $G$ is a Lie group, it is known that $C^\infty_c(G)\subseteq A(G)$ (see \cite[Proposition~3.26]{MR0228628}). 

Since the uniform norm is bounded by the Fourier--Stieltjes norm, it follows that $A(G) \subseteq B(G) \cap C_0(G)$. For many groups this inclusion is strict (see e.g \cite{K-fourier}), but in some cases it is not. We will need the following result when proving Theorem~\ref{thm:rank-one}.

\begin{thm}[{\cite[Theorem~3]{K-fourier}}]\label{thm:fourier-P}
Let $G$ be one of the groups $\SO(1,n)$, $\SU(1,n)$, $\Sp(1,n)$ or $\FF$, and let $G = KAN$ be the Iwasawa decomposition. The group $N$ is contained in a closed amenable group $P$ satisfying $A(P) = B(P)\cap C_0(P)$.
\end{thm}

We will need the following lemma in Section~\ref{sec:SL2}. For a demonstration, see the proof of Proposition~1.12 in \cite{MR784292}.

\begin{lem}[\cite{MR784292}]\label{lem:A-extension}
Let $G$ be a locally compact group with a closed subgroup $H\subseteq G$. If $ u\in A(G)$, then $ u|_H \in A(H)$. Moreover, $\| u|_H \|_{A(H)} \leq \| u\|_{A(G)}$. Conversely, if $ u\in A(H)$, then there is $\tilde u\in A(G)$ such that $ u = \tilde u|_H$ and $\| u\|_{A(H)} = \inf\{ \|\tilde u\|_{A(G)} \mid \tilde u\in A(G),\ \tilde u|_H =  u\}$.
\end{lem}

We now recall the definition of the \emph{Herz--Schur multiplier algebra} $B_2(G)$. A function $k\colon G\times G\to\C$ is a Schur multiplier on $G$ if for every bounded operator $A = [a_{xy}]_{x,y\in G} \in B(\ell^2(G))$ the matrix $[k(x,y)a_{xy}]_{x,y\in G}$ represents a bounded operator on $\ell^2(G)$, denoted $m_k(A)$. If this is the case, then by the closed graph theorem $m_k$ defines a \emph{bounded} operator on $B(\ell^2(G))$, and the Schur norm $\|k\|_S$ is defined as the operator norm of $m_k$.

A continuous function $u\colon G\to\C$ is a Herz--Schur multiplier, if $k(x,y) = u(y^{-1}x)$ is a Schur multiplier on $G$, and the Herz--Schur norm $\|u\|_{B_2}$ is defined as $\|k\|_S$. We let $B_2(G)$ denote the space of Herz--Schur multipliers, which is a Banach space, in fact a unital Banach algebra, with the Herz--Schur norm $\|\ \|_{B_2}$. The Herz--Schur norm dominates the uniform norm.

It is known that $B(G) \subseteq B_2(G)$, and $\|u\|_{B_2} \leq \|u\|_B$ for every $u\in B(G)$. In \cite[Theoreme~1(ii)]{MR0425511}, it is shown that $B_2(G)$ multiplies the Fourier algebra $A(G)$ into itself and $\|uv\|_A \leq \|u\|_{B_2}\|v\|_A$ for every $u\in B_2(G)$, $v\in A(G)$. In this way, we can view $B_2(G)$ as bounded operators on $A(G)$, and $B_2(G)$ inherits a point-norm (or strong operator) topology and a point-weak (or weak operator) topology.

It is known that the space of Herz--Schur multipliers coincides isometrically with the completely bounded Fourier multipliers, usually denoted $M_0A(G)$ (see \cite{MR753889} or \cite{MR1180643}). It is well known that if $G$ is amenable then $B(G) = B_2(G)$ isometrically. The converse is known to hold, when $G$ is discrete (see \cite{MR806070}).

Given $f\in L^1(G)$ and $u\in B_2(G)$ define
\begin{align}\label{eq:duality-Q}
\la f,u\ra = \int_G f(x)u(x)\ dx
\end{align}
and
$$
\|f\|_Q = \sup \{ |\la f,u\ra| \mid u\in B_2(G),\ \|u\|_{B_2} \leq 1 \}.
$$
Then $\|\ \|_Q$ is a norm on $L^1(G)$, and the completion of $L^1(G)$ with respect to this norm is a Banach space $Q(G)$ whose dual space is identified with $B_2(G)$ via \eqref{eq:duality-Q} (see \cite[Proposition~1.10(b)]{MR784292}). In this way $B_2(G)$ is equipped with a weak$^*$-topology coming from its predual $Q(G)$. The weak$^*$-topology is also denoted $\sigma(B_2,Q)$.

We recall that $G$ has the Approximation Property (AP) if there is a net $(u_\alpha)_{\alpha\in A}$ in $B_2(G)\cap C_c(G)$ which converges to the constant function $1_G$ in the $\sigma(B_2,Q)$-topology. As with weak amenability, the definition of the AP just given can be seen to be equivalent to the original definition by use of the convolution trick (see \cite[Appendix~B]{K-WH}). For more on the $\sigma(B_2,Q)$-topology and the AP we refer to the original paper \cite{MR1220905}.

The following lemma is a variant of \cite[Proposition~1.3 (a)]{MR1220905}. The statement of \cite[Proposition~1.3 (a)]{MR1220905} involves an infinite dimensional Hilbert space $\ms H$, but going through the proof of \cite[Proposition~1.3 (a)]{MR1220905} one can check that the statement remains true, if $\ms H$ is just the one-dimensional space $\C$. Hence, we have the following lemma.
\begin{lem}[\cite{MR1220905}]
Let $G$ be a locally compact group. Suppose $a\in L(G)$, $v\in A(G)$ and that $f\in A(G)$ is a compactly supported, positive function with integral 1. Then the functional $\omega_{a,v,f} : B_2(G)\to\C$ defined as
$$
\omega_{a,v,f}(u) = \la a, (f*u)v\ra, \quad u\in B_2(G)
$$
is bounded, that is, $\omega_{a,v,f}\in Q(G)$.
\end{lem}


It is known that weakly amenable groups have the AP  \cite[Theorem~1.12]{MR1220905}, and extensions of groups with the AP have the AP \cite[Theorem~1.15]{MR1220905}. In particular, the group $\R^2\rtimes\SL(2,\R)$ has the AP.

Given a compact subgroup $K$ of $G$ we say that a continuous function $f\colon G\to\C$ is $K$-bi-invariant, if $f(kx) = f(xk) = f(x)$ for every $k\in K$ and $x\in G$. The space of continuous $K$-bi-invariant functions on $G$ is denoted $C(K\backslash G/K)$.

The following two lemmas concerning weak amenability and the AP are standard averaging arguments. For the convenience of the reader, we include a proof of the second lemma. A proof of the first can be manufactured in basically the same way. We note that the special cases where $K$ is the trivial subgroup follow from \cite[Proposition~1.1]{MR996553} and \cite[Theorem~1.11]{MR1220905}, respectively.

\begin{lem}\label{lem:WA}
Let $G$ be a locally compact group with compact subgroup $K$. If $G$ is weakly amenable, say $\Lambda_\WA(G) \leq C$, then there is a net $(v_\beta)$ in $A(G)\cap C_c(K\backslash G/K)$ such that
$$
\| v_\beta v -  v\|_{A(G)} \to 0 \quad\text{for every } v\in A(G)
$$
and $\sup_\beta \| v_\beta\|_{B_2} \leq C$. Moreover, if $G$ is a Lie group, we may arrange that each $v_\beta$ is smooth.
\end{lem}
%
%

\begin{lem}\label{lem:AP}
Let $G$ be a locally compact group with compact subgroup $K$. If $G$ has the AP, then there is a net $(v_\beta)$ in $A(G)\cap C_c(K\backslash G/K)$ such that
$$
\| v_\beta v -  v\|_{A(G)} \to 0 \quad\text{for every } v\in A(G).
$$
Moreover, if $G$ is a Lie group, we may arrange that each $v_\beta$ is smooth.
\end{lem}
\begin{proof}
We suppose $G$ has the AP. Then there is a net $(u_\alpha)$ in $A(G) \cap C_c(G)$ such that $u_\alpha\to 1$ in the $\sigma(B_2,Q)$-topology (see \cite[Remark~1.2]{MR1220905}). Choose a positive function $f\in A(G)$ with compact support and integral 1. By averaging from left and right over $K$ (see Appendix B in \cite{K-WH}), we may further assume that $f$ and each $u_\alpha$ is $K$-bi-invariant. Let $w_\alpha = f * u_\alpha$. Then $w_\alpha\in A(G)\cap C_c(K\backslash G/K)$.

Given $a\in L(G)$ and $v\in A(G)$ we have the following equation:
$$
\la a, w_\alpha v\ra = \omega_{a,v,f}(u_\alpha) \to \omega_{a,v,f}(1) = \la a,v\ra.
$$
Hence $w_\alpha\to 1$ in the point-weak topology on $B_2(G)$. It follows from \cite[Corollary~VI.1.5]{MR0117523} that there is a net $(v_\beta)$ where each $v_\beta$ lies in the convex hull of $\{ w_\alpha\}$ such that $v_\beta \to 1$ in the point-norm topology. In other words, there is a net $(v_\beta)$ in $A(G)\cap C_c(K\backslash G/K)$ such that
$$
\| v_\beta v -  v\|_{A(G)} \to 0 \quad\text{for every } v\in A(G).
$$
If $G$ is a Lie group, we may further assume that $f\in C^\infty_c(G)$, in which case $v_\beta$ becomes smooth.
\end{proof}

\section{Simple Lie groups of higher real rank}\label{sec:higher-rank}

It is known that a connected simple Lie group of real rank at least two is not weakly amenable \cite{MR1418350},\cite{UH-preprint}. In fact, an even stronger result was proved recently \cite{MR2838352}, \cite{MR3047470},\cite{delaat2}. One could ask if such Lie groups also fail the weak Haagerup property. Using results from \cite{MR3047470},\cite{delaat2} we completely settle this question in the affirmative. We thus prove Theorems~\ref{thm:SL3} and Theorem~\ref{thm:simple-Lie-groups}.

\subsection{Three groups of real rank two}
We will prove that the three groups $\SL(3,\R)$, $\Sp(2,\R)$ and the universal covering group $\tilde\Sp(2,\R)$ of $\Sp(2,\R)$ do not have the weak Haagerup property. The cases of $\SL(3,\R)$ and $\Sp(2,\R)$ are similar and are treated together. The case of $\tilde\Sp(2,\R)$ is more difficult, essentially because $\tilde\Sp(2,\R)$ is not a matrix Lie group, and we will go into more details in this case.

When we consider the special linear group $\SL(3,\R)$, then $K = \SO(3)$ will be its maximal compact subgroup. We now describe the group $\Sp(2,\R)$ and a maximal compact subgroup. Consider the matrix $4\times 4$ matrix
$$
J = \begin{pmatrix}
0 & I_2 \\
- I_2 & 0
\end{pmatrix}
$$
where $I_2$ denotes the $2\times 2$ identity matrix. The symplectic group $\Sp(2,\R)$ is defined as
$$
\Sp(2,\R) = \{ g\in \GL(4,\R) \mid g^t J g = J \}.
$$
Here $g^t$ denotes the transpose of $g$. The symplectic group $\Sp(2,\R)$ is a connected simple Lie group of real rank two. It has a maximal compact subgroup
\begin{align}\label{eq:max-SP2}
K = \left\{ \begin{pmatrix}
A & -B \\
B & A
\end{pmatrix} \in M_4(\R) \mid A + iB \in \mr U(2) \right\}
\end{align}
which is isomorphic to $\mr U(2)$.

The following is immediate from \cite[Proposition~4.3, Lemma~A.1(2)]{K-WH}.
\begin{lem}\label{lem:SL3}
Let $G$ be locally compact group with a compact subgroup $K$. If $G$ has the weak Haagerup property, then there is a bounded net $(u_\alpha)$ in $B_2(G)\cap C_0(K\backslash G /K)$ such that $u_\alpha \to 1$ in the weak$^*$-topology.
\end{lem}

We remind the reader that $B_2(G)$ coincides isometrically with the completely bounded Fourier multipliers $M_0A(G)$. The following result is then extracted from \cite[p.~937 + 957]{MR3047470}.
\begin{thm}[\cite{MR3047470}]\label{thm:HdL1}
If $G$ is one of the groups $\SL(3,\R)$ or $\Sp(2,\R)$ and $K$ is the corresponding maximal compact subgroup in $G$, then $B_2(G) \cap C_0(K\backslash G /K)$ is closed in $B_2(G)$ in the weak$^*$-topology.
\end{thm}

\begin{thm}\label{thm:WH-SL3}
The groups $\SL(3,\R)$ and $\Sp(2,\R)$ do not have the weak Haagerup property.
\end{thm}
\begin{proof}
Let $G$ be one of the groups $\SL(3,\R)$ or $\Sp(2,\R)$. Obviously, $1 \notin B_2(G) \cap C_0(K\backslash G /K)$. Since $B_2(G) \cap C_0(K\backslash G /K)$ is weak$^*$-closed, there can be no net $u_\alpha \in B_2(G) \cap C_0(K\backslash G /K)$ such that $u_\alpha \to 1$ in the weak$^*$-topology. Using Lemma~\ref{lem:SL3}, we conclude that $G$ does not have the weak Haagerup property.
\end{proof}

\begin{rem}
An alternative proof of Theorem~\ref{thm:WH-SL3} for the group $\SL(3,\R)$, avoiding the use of the difficult Theorem~\ref{thm:HdL1}, is to use the fact that $\R^2\rtimes\SL(2,\R)$ is a closed subgroup of $\SL(3,\R)$. From Theorem~\ref{thm:SL2} (to be proved in Section~\ref{sec:SL2}), we know that $\R^2\rtimes\SL(2,\R)$ does not have the weak Haagerup property, and this is sufficient to conclude that $\SL(3,\R)$ also fails to have the weak Haagerup property (see \cite[Theorem~A(1)]{K-WH}).
\end{rem}

We now turn to the case of $\tilde\Sp(2,\R)$. To ease notation a bit, in the rest of this section we let $G = \Sp(2,\R)$ and $\tilde G = \tilde\Sp(2,\R)$. We now describe the group $\tilde G$. This is based on \cite{MR2944792} and \cite[Section~3]{delaat2}.

By definition, $\tilde G$ is the universal covering group of $G$. The group $G$ has fundamental group $\pi_1(G) \simeq \pi_1(\mr U(2))$ which is the group $\Z$ of integers. There is a smooth function $c:G\to \mb T$, where $\mb T$ denotes the unit circle in $\C$, such that $c$ induces an isomorphism of the fundamental groups of $G$ and $\mb T$ (such a $c$ is called a circle function). An explicit description of $c$ can be found in \cite{MR2944792} and \cite{delaat2}. The circle function $c$ satisfies
$$
c(1) = 1 \quad\text{and}\quad c(g^{-1}) = c(g)^{-1}.
$$
There is a unique smooth map $\eta:G\times G\to \R$ such that
$$
c(g_1g_2) = c(g_1)c(g_2)e^{i\eta(g_1,g_2)}  \quad\text{and}\quad \eta(1,1) = 0
$$
for all $g_1,g_2\in G$. The map $\eta$ is also explicitly described in \cite{MR2944792} and \cite{delaat2}. The universal cover $\tilde G$ of $G$ can be realized as the smooth manifold
$$
\tilde G = \{(g,t) \in G\times \R \mid c(g) = e^{it} \}
$$
with multiplication given by
$$
(g_1,t_1)(g_2,t_2) = ( g_1g_2 , t_1 + t_2 + \eta(g_1,g_2)).
$$
The identity in $\tilde G$ is $(1,0)$, where $1$ denotes the identity in $G$, and the inverse is given by $(g,t)^{-1} = (g^{-1},-t)$. The map $\sigma:\tilde G\to G$ given by $\sigma(g,t) = g$ is the universal covering homomorphism, and the kernel of $\sigma$ is $\{(1,2\pi k) \in G\times\R \mid k\in\Z\}$, which is of course isomorphic to $\Z$.

Let $K$ be the maximal compact subgroup of $G$ given in \eqref{eq:max-SP2}. Then one can show that
\begin{align}\label{eq:eta-U}
\eta(g,h) = 0 \quad\text{for all } g,h\in K.
\end{align}
Under the obvious identification $K\simeq\mr U(2)$, we consider $\SU(2) \subseteq \mr U(2)$ as a subgroup of $K$. Define a compact subgroup $\tilde H$ of $\tilde G$ by
$$
\tilde H = \{(g,0) \in G\times\R \mid g\in \SU(2) \}.
$$
By \eqref{eq:eta-U} $\tilde H$ is indeed a subgroup of $\tilde G$.

When $t\in\R$ let $v_t \in G$ be the element
$$
v_t = \begin{pmatrix}
\cos t & 0 & -\sin t & 0 \\
0 & \cos t & 0 & -\sin t \\
\sin t & 0 & \cos t & 0 \\
0 & \sin t & 0 & \cos t
\end{pmatrix},
$$
and define $\tilde v_t = (v_t,2t) \in \tilde G$. Then $\eta(v_t,g) = \eta(g,v_t) = 0$ for any $g\in G$. Obviously, $(\tilde v_t)_{t\in\R}$ is a one-parameter family in $\tilde G$, and it is a simple matter to check that conjugation by $\tilde v_t$ is $\pi$-periodic. A simple computation will also show that if $g\in K$, then $gv_t = v_tg$ and hence $h\tilde v_t = \tilde v_t h$ for every $h\in\tilde H$.

Consider the subspace $\mc C$ of $C(\tilde G)$ defined by
$$
\mc C = \{ u\in C(\tilde G) \mid u \text{ is } \tilde H\text{-bi-invariant and } u(\tilde v_t g \tilde v_t^{-1}) = u(g) \text{ for all } t\in\R  \}.
$$
Further, we let $\mc C_0 = \mc C \cap C_0(\tilde G)$. For any $f\in C(\tilde G)$ or $f\in L^1(\tilde G)$, let $f^{\mc C}:\tilde G \to \C$ be defined by
$$
f^{\mc C}(x) = \frac{1}{\pi} \int_0^\pi\!\! \int_{\tilde H} \int_{\tilde H} f(h_1\tilde v_t x \tilde v_t^{-1} h_2) \ dh_1 dh_2 dt,
\quad x\in \tilde G,
$$
where $dh_1$ and $dh_2$ both denote the normalized Haar measure on the compact group $\tilde H$.

\begin{lem}\label{lem:natural}
With the notation as above the following holds:

\begin{enumerate}
	\item If $u\in C(\tilde G)$, then $u^{\mc C} \in \mc C$.
	\item If $f\in L^1(\tilde G)$, then $f^{\mc C}\in L^1(\tilde G)$ and $\|f^{\mc C}\|_Q \leq \|f\|_Q$.
	\item If $u\in B_2(\tilde G)$, then $u^{\mc C}\in B_2(\tilde G)$ and $\|u^{\mc C}\|_{B_2} \leq \|u\|_{B_2}$.
	\item If $u\in C_0(\tilde G)$, then $u^{\mc C} \in C_0(\tilde G)$.
\end{enumerate}
\end{lem}
\begin{proof}
\mbox{}

(1) This is elementary.

(2) Suppose $f\in L^1(\tilde G)$. Connected simple Lie groups are unimodular (see \cite[Corollary~8.31]{MR1920389}), and hence each left or right translate of $f$ is also in $L^1(\tilde G)$ with the same norm. Since left and right translation on $L^1(G)$ is norm continuous, it now follows from usual Banach space integration theory that $f^{\mc C} \in L^1(\tilde G)$.

We complete the proof of (2) after we have proved (3).

(3) This statement is implicit in \cite{delaat2} in the proof of \cite[Lemma~3.10]{delaat2}. We have chosen to include a proof.

For each $g\in B_2(\tilde G)$ or $g\in L^1(\tilde G)$ and $\alpha = (h_1,h_2,t)\in \tilde H\times\tilde H\times\R$ define
$$
g_\alpha(x) = g(h_1\tilde v_t x \tilde v_t^{-1} h_2), \quad x\in \tilde G.
$$
If $g\in B_2(\tilde G)$, then $g_\alpha\in B_2(\tilde G)$ and $\|g_\alpha\|_{B_2} = \|g\|_{B_2}$. Similarly, if $g\in L^1(\tilde G)$, then $g_\alpha\in L^1(\tilde G)$ and $\|g_\alpha\|_1 = \|g\|_1$. Note also that $\la g,f_\alpha \ra = \la g_{\alpha^{-1}} , f\ra$ where $\alpha^{-1} = (h_1^{-1},h_2^{-1}, -t)$. In particular,
$$
|\la g,u_\alpha \ra - \la g,u_\beta\ra| \leq \|g_{\alpha^{-1}} - g_{\beta^{-1}} \|_1 \|u\|_{B_2}
$$
for $\alpha,\beta \in \tilde H\times\tilde H\times\R$, and $\alpha\mapsto u_\alpha$ is weak$^*$-continuous.

The set $S = \{ u_\alpha \mid \alpha\in \tilde H\times\tilde H\times[0,\pi] \}$ is a norm bounded subset of $B_2(\tilde G)$. If $T = \overline{\mr{conv}}^{\sigma(B_2,Q)}(S)$ is the weak$^*$-closed convex hull of $S$, then $T$ is weak$^*$-compact by Banach--Alaoglu's Theorem. By \cite[Theorem~3.27]{MR1157815} the integral
$$
u^{\mc C} = \frac{1}{\pi}\int_{\tilde H\times\tilde H\times[0,\pi]} u_\alpha \ d\mu(\alpha)
$$
exists in $B_2(\tilde G)$. Here $d\mu(\alpha) = dh_1dh_2dt$. Since the set $T$ is bounded in norm by $\|u\|_{B_2}$, and because it follows from \cite[Theorem~3.27]{MR1157815} that $u^{\mc C}\in T$, we obtain the inequality $\|u^{\mc C}\|_{B_2} \leq \|u\|_{B_2}$.


(2) Continued. Let $u\in B_2(\tilde G)$ be arbitrary. Observe that $\la f^{\mc C}, u\ra = \la f,u^{\mc C} \ra$. Hence the norm estimate $\|f^{\mc C}\|_Q \leq \|f\|_Q$ follows from (3).

(4) This is elementary.
\end{proof}

\begin{prop}\label{prop:Sp2}
If $\tilde G$ had the weak Haagerup property, then there would exist a bounded net $(v_\alpha)$ in $B_2(\tilde G) \cap \mc C_0$ such that $v_\alpha \to 1$ in the weak$^*$-topology.
\end{prop}
\begin{proof}
We suppose $\tilde G$ has the weak Haagerup property. Using \cite[Proposition~4.2]{K-WH} we see that there exist a constant $C > 0$ and a net $(u_\alpha)$ in $B_2(\tilde G)\cap C_0(\tilde G)$ such that
$$
\|u_\alpha\|_{B_2} \leq C  \quad\text{ for every } \alpha,
$$
$$
u_\alpha\to 1 \text{ in the }\sigma(B_2,Q)\text{-topology}.
$$
Let $u_\alpha^{\mc C}$ be given by
$$
u_\alpha^{\mc C}(x) = \frac{1}{2\pi} \int_0^{2\pi}\!\!\! \int_{\tilde H\times \tilde H} u_\alpha(h_1\tilde v_t g \tilde v_t^{-1} h_2) \ dh_1 dh_2 dt, \quad x\in \tilde G,
$$
where $dh_1$ and $dh_2$ both denote the normalized Haar measure on $\tilde H$. By Lemma~\ref{lem:natural} we see that $u_\alpha^{\mc C} \in B_2(\tilde G) \cap \mc C_0$ and that $(u_\alpha^{\mc C})$ is a bounded net. Thus, it suffices to prove that $u_\alpha^{\mc C} \to 1$ in the weak$^*$-topology.

By Lemma~\ref{lem:natural}, the map $L^1(\tilde G)\to L^1(\tilde G)$ given by $f\mapsto f^{\mc C}$ extends uniquely to a linear contraction $R:Q(\tilde G) \to Q(\tilde G)$. The dual operator $R^* : B_2(\tilde G)\to B_2(\tilde G)$ obviously satisfies $R^* v = v^{\mc C}$ and is weak$^*$-continuous. Hence
$$
\la f , u_\alpha^{\mc C} \ra = \la f,R^*u_\alpha\ra \to \la f,R^* 1\ra = \la f, 1\ra
$$
for any $f\in Q(G)$. This proves that $u_\alpha^{\mc C} \to 1$ in the weak$^*$-topology.
\end{proof}

For $\beta \geq \gamma \geq 0$, we let $D(\beta,\gamma)$ denote the element in $G$ given as
$$
D(\beta,\gamma) = \begin{pmatrix}
e^\beta & 0 & 0 & 0 \\
0 & e^\gamma & 0 & 0 \\
0 & 0 & e^{-\beta} & 0 \\
0 & 0 & 0 & e^{-\gamma} \\
\end{pmatrix}.
$$
We define $\tilde D(\beta,\gamma)$ as the element $(D(\beta,\gamma),0)$ in $\tilde G$. Let $u\in B_2(\tilde G) \cap \mc C$ be given. If we put
$$
\dot u(\beta,\gamma,t) = u(\tilde v_{\frac t2} \tilde D(\beta,\gamma)),
$$
then it is shown in \cite[Proposition~3.11]{delaat2} that the limit $\lim_{s\to\infty} \dot u (2s,s,t)$ exists for any $t\in\R$. If we let
$$
\mc T = \{ u \in B_2(\tilde G) \cap \mc C \mid \lim_{s\to\infty} \dot u (2s,s,t) = 0 \text{ for all }t\in\R \},
$$
then we can phrase part of the main result of \cite{delaat2} in the following way.
\begin{lem}[{\cite[Lemma~3.12]{delaat2}}]\label{lem:delaat2}
The space $\mc T$ is closed in the weak$^*$-topology.
\end{lem}

Using Lemma~\ref{lem:delaat2}, it is not hard to show that $\tilde G$ does not have the weak Haagerup property. The argument goes as follows.

Obviously, $1\notin \mc T$. We claim that $B_2(\tilde G) \cap \mc C_0 \subseteq \mc T$. Indeed, if $u\in C_0(\tilde G)$ and $t\in\R$, then
$$
\dot u(2s,s,t) = u(\tilde v_{\frac t2} \tilde D(2s,s)) \to 0 \quad\text{ as } s\to\infty.
$$
Since $\mc T$ is weak$^*$-closed, we conclude by Proposition~\ref{prop:Sp2} that $\tilde G$ does not have the weak Haagerup property.

\begin{thm}
The group $\tilde G = \tilde\Sp(2,\R)$ does not have the weak Haagerup property.
\end{thm}

\subsection{The general case}

Knowing that the three groups $\SL(3,\R)$, $\Sp(2,\R)$, and $\tilde\Sp(2,\R)$ do not have the weak Haagerup property, it is a simple matter to generalize this result to include all connected simple Lie groups of real rank at least two. The idea behind the general case is basically that inside any connected simple Lie group of real rank at least two one can find a subgroup that looks like one of the three mentioned groups. We will make this statement more precise now. The following is certainly well known.

\begin{lem}
Let $G$ be a connected simple Lie group of real rank at least two. Then $G$ contains a closed connected subgroup $H$ locally isomorphic to either $\SL(3,\R)$ or $\Sp(2,\R)$.
\end{lem}
\begin{proof}
Consider a connected simple Lie group $G$ of real rank at least two. It is well known that the Lie algebra of such a group contains one of the Lie algebras $\mf{sl}(3,\R)$ or $\mf{sp}(2,\R)$ (see \cite[Proposition~1.6.2]{MR1090825}). Hence there is a connected Lie subgroup $H$ of $G$ whose Lie algebra is either $\mf{sl}(3,\R)$ or $\mf{sp}(2,\R)$ (see \cite[Theorem~II.2.1]{MR514561}). By \cite[Theorem~II.1.11]{MR514561} we get that $H$ is locally isomorphic to $\SL(3,\R)$ or $\Sp(2,\R)$. It remains only to see that $H$ is closed. This is \cite[Corollary~1]{MR1418350}.
\end{proof}

\begin{repthm}{thm:simple-Lie-groups}
A connected simple Lie group has the weak Haagerup property if and only if it has real rank zero or one.
\end{repthm}
\begin{proof}
It is known that connected simple Lie groups of real rank zero and one have the weak Haagerup property. Indeed, connected simple Lie groups of real rank zero are compact, and connected simple Lie groups of real rank one are weakly amenable (see \cite{MR996553},\cite{MR1079871}). This is clearly enough to conclude that such groups have the weak Haagerup property. Thus, we must prove that connected simple Lie groups of real rank at least two do not have the weak Haagerup property.

Let $G$ be a connected simple Lie group of real rank at least two. Then $G$ contains a closed connected subgroup $H$ locally isomorphic to $\SL(3,\R)$ or $\Sp(2,\R)$. Because of \cite[Theorem~A(1)]{K-WH}, it is sufficient to show that $H$ does not have the weak Haagerup property.

Suppose first that $H$ is locally isomorphic to $\SL(3,\R)$. The fundamental group of $\SL(3,\R)$ has order two, and $\SL(3,\R)$ has trivial center. Hence the universal covering group of $\SL(3,\R)$ has center of order two, and $H$ must have finite center $Z$ of order one or two. Then $\SL(3,\R) \simeq H/Z$. Since $\SL(3,\R)$ does not have the weak Haagerup property, we deduce from \cite[Theorem~A(2)]{K-WH} that $H$ does not have the weak Haagerup property.

Suppose instead that $H$ is locally isomorphic to $\Sp(2,\R)$. Then there is a central subgroup $Z \subseteq \tilde\Sp(2,\R)$ such that $\tilde\Sp(2,\R) / Z \simeq H$. Since the center of $\tilde\Sp(2,\R)$ is isomorphic to $\pi_1(\Sp(2,\R))\simeq\Z$, every nontrivial subgroup of the center of $\tilde\Sp(2,\R)$ is infinite and of finite index. Hence, if $H$ has infinite center, then $\tilde\Sp(2,\R)\simeq H$. In that case, $H$ does not have the weak Haagerup property. Otherwise, $H$ has finite center $Z$, and then $H/Z \simeq \Sp(2,\R)/\{\pm 1\}$. Since $\Sp(2,\R)$ does not have the weak Haagerup property, we deduce from \cite[Theorem~A(2)]{K-WH} that $H$ does not have the weak Haagerup property.

\end{proof}

\section{Simple Lie groups of real rank one}\label{sec:rank-1}
In this section we compute the weak Haagerup constant of the groups $\Sp(1,n)$ and $\FF$. We thus prove Theorem~\ref{thm:rank-one}. Throughout this section $G$ denotes one of the groups $\Sp(1,n)$, $n\geq 2$ or $\FF$. The symbol $\R_+$ denotes the nonnegative reals, that is, $\R_+ = [0,\infty[$.

\subsection{Preparations}

The group $\Sp(1,n)$ is defined as the group of quaternion matrices of size $n+1$ that preserve the Hermitian form
$$
\la x,y \ra = \bar y_1 x_1 - \sum_{k=2}^{n+1} \bar y_k x_k, \quad x = (x_k)_{k=1}^{n+1}, y = (y_k)_{k=1}^{n+1} \in \mb H^{n+1}.
$$
Equivalently,
$$
\Sp(1,n) = \{ g\in \GL(n+1,\mb H) \mid g^* I_{1,n} g = I_{1,n} \}
$$
where $I_{1,n}$ is the $(n+1)\times(n+1)$ diagonal matrix
$$
I_{1,n} = \begin{pmatrix}
1 & & & \\
& -1 & & \\
& & \ddots & \\
& & & -1
\end{pmatrix}.
$$
The exceptional Lie group $\FF$ is described in \cite{MR560851}.

For details about general structure theory of semisimple Lie groups we refer to \cite[Chapters~VI-VII]{MR1920389} and \cite[Chapter~IX]{MR514561}. The proof of Theorem~\ref{thm:rank-one} builds on \cite{MR996553}, where \eqref{eq:WA} is proved. We adopt the following from \cite{MR996553}.

Recall that throughout this section $G$ denotes one of the connected simple real rank one Lie groups $\Sp(1,n)$, $n\geq 2$ or $\FF$. Let $\mf g$ be the Lie algebra of $G$. Let $\theta$ be a Cartan involution, $\mf g = \mf k \oplus \mf p$ the corresponding Cartan decomposition and $K$ the analytic subgroup corresponding to $\mf k$. Then $K$ is a maximal compact subgroup of $G$. Let $\mf a$ be a maximal abelian subalgebra of $\mf p$, and decompose $\mf g$ into root spaces,
$$
\mf g = \mf m \oplus \mf a \oplus \sum_{\beta\in\Sigma} \mf g_\beta,
$$
where $\mf m$ is the centralizer of $\mf a$ in $\mf k$ and $\Sigma$ is the set of roots. Then $\mf a$ is one dimensional and $\Sigma = \{-2\alpha,-\alpha,\alpha,2\alpha\}$. Let $\mf n = \mf g_\alpha \oplus \mf g_{2\alpha}$. We have the Iwasawa decomposition at the Lie algebra level
$$
\mf g = \mf k \oplus \mf a \oplus \mf n
$$
and at the group level
$$
G = KAN
$$
where $A$ and $N$ are the analytic subgroups of $G$ with Lie algebras $\mf a$ and $\mf n$, respectively. The group $A$ is abelian and simply connected, and $N$ is nilpotent and simply connected.

Let $B$ be the Killing form of $\mf g$. Let $\mf v = \mf g_\alpha$, $\mf z = \mf g_{2\alpha}$ and equip the Lie algebra $\mf n = \mf v \oplus \mf z$ with the inner product
$$
\la v+z,v'+z'\ra = \frac{-1}{2p+4q} B\left( \frac v2 + \frac z4, \theta\left( \frac{v'}{2} + \frac{z'}{4} \right)\right)
$$
$v,v'\in\mf v$, $z,z'\in\mf z$ where, as in \cite{MR996553},
$$
2p = \dim\mf v,
\quad
q = \dim\mf z.
$$
The inner product on $\mf n$ of course gives rise to a norm $|\ \ |$ on $\mf n$ defined by $|n|~=~\sqrt{\la n,n\ra}$, $n\in\mf n$.

The following convenient notation is taken from \cite{MR996553}. Let
\begin{align}\label{eq:vz}
(v,z) = \exp(v + z/4), \quad v\in\mf v, \ z\in\mf z.
\end{align}
Then $(v,z)\in N$. Since $N$ is connected, nilpotent and simply connected, the exponential mapping is a diffeomorphism of $\mf n$ onto $N$ (\cite[Theorem~1.127]{MR1920389}), and hence every element of $N$ can in a unique way be written in the form \eqref{eq:vz}.

We let $a = p/2$. It is well known that the values of $p$, $q$, and $a$ are as follows:
\renewcommand{\arraystretch}{1.2}
\setlength{\tabcolsep}{10pt}
\begin{align}\label{eq:constants}
\begin{tabular}{@{}lccc@{}}
\hline
Group & $p$ & $q$ & $a$ \\\hline
$\Sp(1,n)$ & $2n-2$ & $3$ & $n-1$ \\ 
$\FF$ & $4$ & $7$ & $2$ \\\hline
\end{tabular}
\end{align}

As $\mf a$ is one-dimensional there is a unique element $H$ in $\mf a$ such that $\mr{ad}(H)|_{\mf g_\alpha} = \id_{\mf g_\alpha}$. Let
$$
a_t = \exp(t H) \in A, \quad t\in\R,
$$
and $\overline{A^+} = \overline{\{ \exp tH \mid t > 0\}} = \{a_t\in A \mid t\geq 0\}$. Then we have the $KAK$ decomposition of $G$ (see \cite[Theorem~IX.1.1]{MR514561})
\begin{align}\label{eq:KAK}
G = K\overline{A^+}K.
\end{align}

More precisely, for each $g\in G$ there is a unique $t\geq 0$ such that $g\in Ka_tK$. Concerning the $KAK$ decomposition of elements of $N$ we can be even more specific. The following lemma is completely analogous to part of \cite[Proposition~2.1]{MR996553}, and thus we leave out the proof.
\begin{lem}\label{lem:KAK}
For every $v\in\mf v$ and $z\in\mf z$ exists a unique $t\in\R_+$ such that
$$
(v,z) \in K a_t K.
$$
Moreover, $t$ satisfies
$$
4\sinh^2 t = 4|v|^2 + |v|^4 + |z|^2.
$$
\end{lem}

The following fact proved by Whitney \cite[Theorem~1]{MR0007783} identifies the smooth even functions on $\R$ with smooth functions on $\R_+ = [0,\infty[$.

\begin{lem}[\cite{MR0007783}]\label{lem:whitney}
An even function $g$ on $\R$ is smooth if and only if it has the form $g(x) = f(x^2)$ for some (necessarily unique) $f\in C^\infty(\R_+)$.
\end{lem}

The following proposition is inspired by Theorem 2.5(b) in \cite{MR996553}.
\begin{prop}\label{prop:K-bi-inv}
Suppose $u\in C(N)$. Then $u$ is the restriction to $N$ of a $K$-bi-invariant function on $G$ if and only if $u$ is of the form
\begin{align}\label{eq:u-form}
(v,z) \mapsto f(4|v|^2 + |v|^4 + |z|^2)
\end{align}
for some $f\in C(\R_+)$. In that case, the function $f$ is uniquely determined by $u$.

The function $f$ is in $C^\infty(\R_+)$, $C_c(\R_+)$, or $C_0(\R_+)$ if and only if $u$ is in $C^\infty(N)$, $C_c(N)$, or $C_0(N)$, respectively.
\end{prop}
\begin{proof}
Assume $u\in C(N)$ is the restriction to $N$ of a $K$-bi-invariant function on $G$. Then by Lemma~\ref{lem:KAK}, $u(v,z)$ only depends on $4|v|^2 + |v|^4 + |z|^2$ when $v\in\mf v$, $z\in\mf z$. Hence there is a unique function $f$ on $\R_+$ such that
$$
u(v,z) = f(4|v|^2 + |v|^4 + |z|^2), \quad v\in\mf v,\ z\in\mf z.
$$
If we fix a unit vector $z_0\in\mf z$ then $t\mapsto u(0,\sqrt t z_0) = f(t)$ is continuous on $\R_+$, since $u$ is continuous. In other words, $f\in C(\R_+)$.

Assume conversely that $u$ is of the form \eqref{eq:u-form} for some (necessarily unique) $f\in C(\R_+)$. We define a function $\tilde u$ on $G$ using the $K\overline{A^+}K$ decomposition as follows. For an element $ka_t k'$ in $G$ where $k,k'\in K$ and $t\in\R_+$ we let
$$
\tilde u(ka_t k') = f(4\sinh^2 t).
$$
By the uniqueness of $t$ in the $K\overline{A^+}K$ decomposition, this is well-defined. Clearly, $\tilde u$ is a $K$-bi-invariant function on $G$. When $(v,z)\in N$ we find by Lemma~\ref{lem:KAK} that
$$
\tilde u(v,z) = f(4|v|^2 + |v|^4 + |z|^2) = u(v,z)
$$
so that $\tilde u$ restricts to $u$ on the subgroup $N$.

It is easy to see that $u$ has compact support if and only if $f$ has compact support, and similarly that $u$ vanishes at infinity if and only if $f$ vanishes at infinity. It is also clear that smoothness of $f$ implies smoothness of $u$.

Finally, assume that $u$ is smooth. If again $z_0\in\mf z$ is a unit vector, then $t\mapsto u(0,tz_0) = f(t^2)$ is a smooth even function on $\R$. By Lemma~\ref{lem:whitney} we obtain $f\in C^\infty(\R_+)$.
\end{proof}

We remark that $\|u\|_\infty = \|f\|_\infty$.

\begin{lem}
Let $(u_k)$ be a sequence, where $u_k\in C(N)$ is the restrictions to $N$ of a $K$-bi-invariant function in $C(G)$, and let $f_k\in C(\R_+)$ be as in Proposition~\ref{prop:K-bi-inv}. If $u_k\to 1$ pointwise, then $f_k\to 1$ pointwise.
\end{lem}
\begin{proof}
This is obvious, since the map $(v,z)\mapsto 4|v|^2 + |v|^4 + |z|^2$ is a surjection of $N$ onto $\R_+$.
\end{proof}

\subsection{Proof of Theorem~\ref{thm:rank-one}}

With almost all the notational preparations in place, we are now ready to aim for the proof of Theorem~\ref{thm:rank-one}. The starting point is the inequality in Proposition~\ref{prop:CH-1} which is taken almost directly from \cite{MR996553}. To ease notation a bit, let
\begin{align}\label{eq:C}
C= \frac{2^{p+1}\Gamma\left(\frac{p+q}{2}\right)}{\Gamma(p)\Gamma\left(\frac{q}{2}\right)}.
\end{align}
We remark that with this definition $C$ and \eqref{eq:WA} and \eqref{eq:constants} in mind, then
\begin{align}\label{eq:Lambda}
\frac C4 \Gamma(a) = \Lambda_\WA(G).
\end{align}
Combining Theorem~2.5(b), Proposition~5.1, and Proposition~5.2 in \cite{MR996553} one obtains the following proposition.
\begin{prop}[\cite{MR996553}]\label{prop:CH-1}
If $u\in C^\infty_c(N)$ is the restriction of a $K$-bi-invariant function on $G$, then $u$ is of the form
$$
u(z,v) = f(4|v|^2 + |v|^4 + |z|^2)
$$
for some $f\in C_c^\infty(\R)$, and
$$
\left| C \int_{\R_+} f^{(a)}(4t^2 + t^4)t^{2p-1} \ dt \right| \leq \|u\|_{A(N)}.
$$
\end{prop}

We now aim to prove a variation of the above proposition where we no longer require the function $u$ to be compactly supported.

Following \cite{MR996553}, we let $h\colon]0,\infty[ \to \R$ be defined by $h(s) = (s^{1/2} - 2)^{p-1} s^{-1/2}$, and let $g$ be the $(a-1)$'th derivative of $h$. It is known (see \cite[p.~544]{MR996553}) that
\begin{align}\label{eq:g}
\int_4^\infty |g'(s)| \ ds < \infty \quad\text{and}\quad \int_4^\infty g'(s) \ ds = \Gamma(a).
\end{align}

\begin{prop}\label{prop:norm-est}
If $u\in A(N) \cap C^\infty(N)$ is the restriction of a $K$-bi-invariant function on $G$, then $u$ is of the form
$$
u(z,v) = f(4|v|^2 + |v|^4 + |z|^2)
$$
for some $f\in C_0^\infty(\R_+)$, and
$$
\left| \frac C4 \int_4^\infty f(s-4)g'(s) \ ds \right| \leq \|u\|_{A(N)}.
$$
\end{prop}
\begin{proof}
We use the fact that $G$ is weakly amenable \cite{MR996553}. We will then approximate $u$ by functions in $C^\infty_c(N)$ and apply Proposition~\ref{prop:CH-1} to those functions.

Choose a sequence $v_k \in C^\infty_c(G)$ of $K$-bi-invariant functions such that $\|v_k\|_{B_2} \leq \Lambda_\WA(G)$ and
$$
\| v_k v - v \|_{A(G)} \to 0 \quad \text{ for every } v\in A(G)
$$
(see Lemma~\ref{lem:WA}). Put $w_k = v_k|_N$. Then by Lemma~\ref{lem:A-extension} we have
$$
\| w_k v - v \|_{A(N)} \to 0 \quad \text{ for every } v\in A(N).
$$
If we put $u_k = w_k u$, then we get $ u_k\to u$ uniformly. Note that $u_k \in C^\infty_c(N)$. Let $f\in C^\infty_0(\R_+)$ and $f_k \in C^\infty_c(\R_+)$ be chosen according to Proposition~\ref{prop:K-bi-inv} such that
$$
u(v,z) = f(4|v|^2 + |v|^4 + |z|^2), \qquad u_k (v,z) = f_k(4|v|^2 + |v|^4 + |z|^2).
$$
Using the substitution $s = 4 + 4t^2 + t^4$ and then partial integration we get
\begin{align*}
\| u \|_{A(N)}
& = \lim_k \| u_k\|_{A(N)} \\
&\geq \lim_k \left| C \int_{\R_+} f_k^{(a)}(4t^2 + t^4)t^{2p-1} \ dt \right|  \\
&= \lim_k \left| \frac C4 \int_4^\infty f_k^{(a)}(s-4) h(s) \ ds \right|  \\
&= \lim_k \left| \frac C4 \int_4^\infty f_k(s-4)g'(s) \ ds \right|.
\end{align*}
There are no boundary terms, since $f_k$ has compact support, and because the first $p-2$ derivatives of $h$ vanish at $s = 4$. We observe that
$$
\|f_k\|_\infty = \|u_k\|_\infty \leq \|u_k\|_{B_2} \leq \|w_k\|_{B_2} \|u\|_{B_2} \leq \|v_k\|_{B_2} \|u\|_{B_2} \leq \Lambda_\WA(G) \ \|u\|_{B_2},
$$
so in particular, $\sup_k \|f_k\|_\infty < \infty$. Finally, since $g'(s)$ is integrable (see \eqref{eq:g}), we can apply Lebesgue's Dominated Convergence Theorem and get
$$
\| u \|_{A(N)} \geq \left| \frac C4 \int_4^\infty f(s-4)g'(s) \ ds \right|.
$$
\end{proof}

\begin{prop}\label{prop:est3}
Suppose $u_k\in A(N) \cap C^\infty(N)$ is the restriction of a $K$-bi-invariant function on $G$, and suppose further that $u_k\to 1$ pointwise as $k\to\infty$. Then
$$
\sup_k \|u_k\|_{A(N)} \geq \Lambda_\WA(G).
$$
\end{prop}
\begin{proof}
If $\sup_k \| u_k\|_{A(N)} = \infty$, there is nothing to prove. So we assume that $\sup_k \| u_k\|_{A(N)} < \infty$.

Let $f_k \in C^\infty_0(\R_+)$ be chosen according to Proposition~\ref{prop:K-bi-inv} such that
$$
u_k (v,z) = f_k(4|v|^2 + |v|^4 + |z|^2).
$$
Observe that $f_k \to 1$ pointwise, and $\sup_k \|f_k\|_\infty < \infty$. By Lebesgue's Dominated Convergence Theorem we have
\begin{align*}
\sup_k \| u_k \|_{A(N)}
\geq \lim_k \left| \frac C4 \int_4^\infty f_k(s-4)g'(s) \ ds \right|  
= \frac C4 \int_4^\infty g'(s) \ ds
= \frac C4 \ \Gamma(a).
\end{align*}
As mentioned in \eqref{eq:Lambda}, $\Lambda_\WA(G) = C\Gamma(a)/4$.
\end{proof}

Theorem~\ref{thm:rank-one} is an immediate consequence of the following, since we already know the value $\Lambda_\WA(G)$ and that $\Lambda_\WH(G) \leq \Lambda_\WA(G)$.

\begin{prop}\label{prop:rank-one-main-prop}
If $G$ is either $\Sp(1,n)$, $n\geq 2$, or $\FF$, then $\Lambda_\WH(G) = \Lambda_\WA(G)$.
\end{prop}
\begin{proof}
We only prove $\Lambda_\WH(G) \geq \Lambda_\WA(G)$, since the other inequality holds trivially. Using Proposition~4.3 in \cite{K-WH}, it is enough to prove that if a sequence $v_k \in B_2(G) \cap C^\infty_0(G)$ consisting of $K$-bi-invariant functions satisfies
$$
\| v_k\|_{B_2} \leq L \quad \text{for all } k,
$$
$$
v_k \to 1 \quad\text{uniformly on compacts as } k\to\infty,
$$
then $L \geq \Lambda_\WA(G)$. So suppose such a sequence is given. Consider the subgroup $P$ from Theorem~\ref{thm:fourier-P}. Since $P$ is amenable, $B_2(P) = B(P)$ isometrically. Then
$$
v_k|_P \in B_2(P) \cap C_0(P) = B(P) \cap C_0(P) = A(P)
$$
by Theorem~\ref{thm:fourier-P}, and so $v_k|_N \in A(N)$. To ease notation, we let $u_k = v_k|_N$. Then (using amenability of $N$)
$$
\|u_k\|_{A(N)} = \|u_k\|_{B(N)} = \|u_k\|_{B_2(N)} \leq \| v_k\|_{B_2(G)}.
$$
Hence by Proposition~\ref{prop:est3} and the above inequalities we conclude
$$
\Lambda_\WA(G) \leq \sup_k \|u_k\|_{A(N)} \leq L.
$$
This shows that $\Lambda_\WA(G) \leq \Lambda_\WH(G)$, and the proof is complete.
\end{proof}

\begin{proof}[Proof of Theorem~\ref{thm:rank-one}]
Suppose $G$ is a connected simple Lie group. If the real rank of $G$ is zero, then $G$ is compact and $\Lambda_\WA(G) = \Lambda_\WH(G) = 1$. If the real rank of $G$ is at least two, then $\Lambda_\WA(G) = \infty$ by \cite{UH-preprint}, \cite{MR1418350}. By Theorem~\ref{thm:simple-Lie-groups} also $\Lambda_\WH(G) = \infty$.

Only the case when the real rank of $G$ equals one remains. Then $G$ is locally isomorphic to either $\SO_0(1,n)$, $\SU(1,n)$, $\Sp(1,n)$ where $n \geq 2$ or locally isomorphic to $\FF$ (see, for example, the list \cite[p.~426]{MR1920389} and \cite[Theorem~II.1.11]{MR514561}). If $G$ is locally isomorphic to $\SO_0(1,n)$ or $\SU(1,n)$ then by \eqref{eq:WA} we conclude that $\Lambda_\WA(G) = \Lambda_\WH(G) = 1$.

Finally, let $\tilde G$ be either $\Sp(1,n)$ or $\FF$ and suppose $G$ is locally isomorphic to $\tilde G$. If $KAN$ is the Iwasawa decomposition of $\tilde G$ then $K$ is $\Sp(n)\times\Sp(1)$ or $\Spin(9)$, respectively (see Section~4, Proposition~1 and Section~5, Theorem~1 in \cite{MR560851} for the latter). Here $\Spin(9)$ is the two-fold simply connected cover of $\SO(9)$. In any case, $K$ is simply connected and compact, so it follows that $\tilde G$ is simply connected with finite center (\cite[Theorem~6.31]{MR1920389}).

From Proposition~\ref{prop:rank-one-main-prop}, we get that $\Lambda_\WH(G) = \Lambda_\WA(G)$ if $G = \tilde G$. Otherwise $G$ is a quotient of $\tilde G$ by a finite central subgroup, and then it follows from \eqref{eq:WA}, Proposition~\ref{prop:rank-one-main-prop} and \cite[Proposition~5.4]{K-WH} that $\Lambda_\WH(G) = \Lambda_\WH(\tilde G) = \Lambda_\WA(G)$.
\end{proof}

\section{Another group without the weak Haagerup property}\label{sec:SL2}
Throughout this section, we let $G$ be the group $G = \R^2\rtimes\SL(2,\R)$. We show here that this group does not have the weak Haagerup property. In short, we prove $\Lambda_\WH(G) = \infty$. This generalizes a result from \cite{MR1245415} and \cite{UH-preprint}, where it is proved that $\Lambda_\WA(G) = \infty$.

We shall think of $G$ as a subgroup of $\SL(3,\R)$ in the following way:
\begin{align*}
G &= \left(\begin{array}{cc|c}
\multicolumn{2}{c|}{\multirow{2}{*}{$\SL(2,\R)$}} & \multirow{2}{*}{$\R^2$} \\
 & & \\
 \hline
\multicolumn{2}{c|}{0} & 1
\end{array}\right).
\end{align*}
We consider the compact group $K = \SO_2(\R)$ as a subgroup of $G$ using the inclusions
$$
\SO_2(\R) \subseteq \SL(2,\R) \subseteq \R^2\rtimes\SL(2,\R).
$$
We will make use of the following closed subgroups of $G$.
\begin{align}
N = \left\{\begin{pmatrix}
1 & x & z \\
0 & 1 & y \\
0 & 0 & 1
\end{pmatrix}
\middle\vert x,y,z\in\R
\right\},
\qquad
\label{eq:P}
P = \left\{\begin{pmatrix}
\lambda & x & z \\
0 & \lambda^{-1} & y \\
0 & 0 & 1
\end{pmatrix}
\middle\vert x,y,z\in\R,\ \lambda>0
\right\}.
\end{align}
The group $N$ is the Heisenberg group. The following is proved in \cite[Section~10]{MR1245415} (see also \cite[Lemma~A and Lemma~E]{UH-preprint}).

\begin{prop}[{\cite{MR1245415}}]\label{prop:HE}
Consider the Heisenberg group $N$. If $ u\in C^\infty_c(N)$ is the restriction of a $K$-bi-invariant function in $C^\infty(G)$, then
$$
\left| \int_{-\infty}^\infty \frac{ u(x,0,0)}{\sqrt{1 + x^2/4}}\ dx \right| \leq 12\pi \| u\|_{A(N)}.
$$
\end{prop}

We now prove a variation of the above lemma where we no longer require the function in question to be compactly supported.

\begin{prop}\label{prop:norm-integral}
Suppose $ u\in A(N)\cap C^\infty(N)$ is the restriction of a $K$-bi-invariant function in $C^\infty(G)$. Then
$$
\int_{-\infty}^\infty \frac{| u(x,0,0)|^2}{\sqrt{1 + x^2/4}}\ dx  \leq 12\pi \| u\|_{A(N)}^2.
$$
\end{prop}
\begin{proof}
The idea is to use the fact (see \cite[p.~670]{MR1220905}) that $G = \R^2 \rtimes \SL(2,\R)$ has the AP. We will approximate $u$ with compactly supported, smooth functions on $N$ and then apply Proposition~\ref{prop:HE}.

By Lemma \ref{lem:A-extension}, there is an extension $\tilde u\in A(G)$ of $ u$. It follows from Lemma~\ref{lem:AP} that there is a \emph{sequence} $(v_k)$ in $C^\infty_c(K\backslash G/K)$ such that
$$
\|v_k \tilde u - \tilde u\|_{A(G)} \to 0.
$$
We let $w_k = v_k|_N$. Since restriction does not increase the norm, we have
$$
\| w_k u -  u\|_{A(N)} =  \|( v_k\tilde u - \tilde u)|_N\|_{A(N)} \leq \| v_k\tilde u - \tilde u\|_{A(G)} \to 0.
$$
Since $ w_k u \to  u$ pointwise we have by Fatou's Lemma and Proposition~\ref{prop:HE} applied to $| w_k u|^2$
\begin{align*}
\| u\|_{A(N)}^2 \geq \liminf_{k\to\infty} \| | w_k u|^2\|_{A(N)}
&\geq \frac{1}{12\pi} \int_{-\infty}^\infty \liminf_{k\to\infty} \frac{| w_k u(x,0,0)|^2}{\sqrt{1+ x^2/4}} \ dx \\
&= \frac{1}{12\pi} \int_{-\infty}^\infty \frac{| u(x,0,0)|^2}{\sqrt{1+ x^2/4}} \ dx.
\end{align*}
In the first inequality we have used that for every $v\in A(N)$ we have $|v|^2 = v \overline v \in A(N)$ and $\| |v|^2 \|_A \leq \| v\|_A\|\overline v\|_A = \|v\|_A^2$.
\end{proof}



Having done all the necessary preparations, we are now ready for
\begin{repthm}{thm:SL2}
The group $\R^2\rtimes\SL(2,\R)$ does not have the weak Haagerup property.
\end{repthm}
\begin{proof}
Suppose there is a net $( u_n)$ of Herz--Schur multipliers on $G=\R^2\rtimes\SL(2,\R)$ vanishing at infinity and converging uniformly to the constant function $1_G$ on compacts. We will show that $\sup_n \| u_n\|_{B_2} = \infty$. By Proposition~4.3 in \cite{K-WH}, we may assume that $ u_n\in C^\infty_0(K\backslash G/K)$, and since $G$ is second countable, we may assume that the net is a sequence.

Consider the group $P$ defined in \eqref{eq:P}. Since $P$ is amenable, even solvable, we know that $B_2(P) = B(P)$ isometrically. We also know that $A(P) = B(P) \cap C_0(P)$ (see \cite[Theorem~2]{K-fourier}). Then
$$
 u_n|_P \in B_2(P) \cap C_0(P) = B(P) \cap C_0(P) = A(P),
$$
and so $ u_n|_N \in A(N)$. To ease notation, we let $ w_n =  u_n|_N$.
Then, using amenability of $N$,
$$
\| w_n\|_{A(N)} = \| w_n\|_{B(N)} = \| w_n\|_{B_2(N)} \leq \|  u_n\|_{B_2(G)}.
$$
Hence it will suffice to show that $\sup_n \| w_n\|_{A(N)} = \infty$. By Proposition~\ref{prop:norm-integral}, we have
$$
\int_{-\infty}^\infty \frac{| w_n(x,0,0)|^2}{\sqrt{1 + x^2/4}}\ dx  \leq 12\pi \| w_n\|_{A(N)}^2.
$$
Since $ u_n\to 1_G$ uniformly on compacts, we have in particular $ w_n(x,0,0)\to 1$ as $n\to\infty$. It follows that
$$
\liminf_{n\to\infty} \| w_n\|_{A(N)}^2 \geq \frac{1}{12\pi} \int_{-\infty}^\infty \frac{1}{\sqrt{1 + x^2/4}}\ dx = \infty.
$$
This completes the proof.
\end{proof}

\begin{rem}
It was proved by the first author \cite{UH-preprint} that the group $\R^2\rtimes\SL(2,\R)$ is not weakly amenable. This result was later generalized by Dorofaeff \cite{MR1245415} to include the groups $\R^n\rtimes\SL(2,\R)$ where $n\geq 2$. Here the action of $\SL(2,\R)$ on $\R^n$ is by the unique irreducible representation of dimension $n$.

In view of Theorem~\ref{thm:SL2}, and especially since our proof of Theorem~\ref{thm:SL2} uses the same techniques as \cite{UH-preprint} and \cite{MR1245415}, it is natural to ask if the groups $\R^n\rtimes\SL(2,\R)$ also fail to have the weak Haagerup property when $n\geq 3$.

We note that an affirmative answer in the case $n = 3$ would give a different proof of Theorem~\ref{thm:SL3}. This is because $\SL(3,\R)$ contains $\R^2\rtimes\SL(2,\R)$ as a closed subgroup, and both groups $\Sp(2,\R)$ and $\tilde\Sp(2,\R)$ contain $\R^3\rtimes\SL(2,\R)$ as a closed subgroup (see \cite{MR1418350}).

\end{rem}


\end{document}